\documentclass[11pt,article]{amsart}
\usepackage{amsxtra}
\usepackage[all]{xy}
\usepackage{palatino}
\usepackage{commath}
\usepackage{mathrsfs}
\usepackage{comment}
\usepackage{mathtools}

\usepackage{tikz-cd}
\usepackage[bookmarks=false]{hyperref}

\usepackage{amssymb,amsmath,amsthm,epsf,epsfig,dsfont,bbm}
 
\usepackage{upref, eucal}

\begin{document}

\newcommand\GSP{{\mathfrak {GSp}}}
\newcommand\SP{{\mathrm {Sp}}}
\newcommand\A{\mathbb{A}}
\newcommand\G{\mathbb{G}}
\newcommand\N{\mathbb{N}}
\newcommand\T{\mathbb{T}}
\newcommand\sO{\mathcal{O}}
\newcommand\sE{{\mathcal{E}}}
\newcommand\tE{{\mathbb{E}}}
\newcommand\sF{{\mathcal{F}}}
\newcommand\sG{{\mathcal{G}}}
\newcommand\GL{{\mathrm{GL}}}
\newcommand\GS{{\mathrm{GSp}}}
\newcommand\HH{{\mathrm H}}
\newcommand\mM{{\mathrm M}}
\newcommand\fS{\mathfrak{S}}
\newcommand\fP{\mathfrak{P}}
\newcommand\fQ{\mathfrak{Q}}
\newcommand\Qbar{{\bar{\bf Q}}}
\newcommand\sQ{{\mathcal{Q}}}
\newcommand\sP{{\mathbb{P}}}
\newcommand{\Q}{{\bf Q}}
\newcommand{\tH}{\mathbb{H}}
\newcommand{\Z}{{\bf Z}}
\newcommand{\R}{{\bf R}}
\newcommand{\C}{{\bf C}}
\newcommand{\g}{{\mathfrak{g}}}
\newcommand{\F}{\mathbb{F}}
\newcommand\gP{\mathfrak{P}}
\newcommand\Gal{{\mathrm {Gal}}}
\newcommand\SL{{\mathrm {SL}}}
\newcommand\SO{{\mathrm {SO}}}
\newcommand\gl{{\mathfrak {gl}}}
\newcommand\Gsp{{\mathrm {Gsp}}}
\newcommand\So{{\mathrm {Sp}}}
\newcommand\Sl{{\mathfrak {sl}}}
\newcommand\Sp{{\mathfrak {sp}}}
\newcommand\Hom{{\mathrm {Hom}}}
\newcommand{\legendre}[2] {\left(\frac{#1}{#2}\right)}
\newcommand\iso{{\> \simeq \>}}
\newcommand\Frob{{\mathrm {Frob}}}
\newtheorem{thm}{Theorem}
\newtheorem{theorem}[thm]{Theorem}
\newtheorem{cor}[thm]{Corollary}
\newtheorem{conj}[thm]{Conjecture}
\newtheorem{prop}[thm]{Proposition}
\newtheorem{lemma}[thm]{Lemma}
\theoremstyle{definition}
\newtheorem{definition}[thm]{Definition}
\theoremstyle{remark}
\newtheorem{remark}[thm]{Remark}
\newtheorem{example}[thm]{Example}
\newtheorem{claim}[thm]{Claim}

\newtheorem{lem}[thm]{Lemma}

\theoremstyle{definition}
\newtheorem{dfn}{Definition}

\theoremstyle{remark}
\setlength{\abovedisplayskip}{2pt}
\setlength{\belowdisplayskip}{2pt}

\theoremstyle{remark}
\newtheorem*{fact}{Fact}
\makeatletter
\def\imod#1{\allowbreak\mkern10mu({\operator@font mod}\,\,#1)}
\makeatother
\theoremstyle{remark}
\makeatletter
\def\imod#1{\allowbreak\mkern10mu({\operator@font mod}\,\,#1)}
\makeatother
\title{Siegel modular forms with extra twists}

\author{Debargha Banerjee and  Ronit Debnath}
\address{INDIAN INSTITUTE OF SCIENCE EDUCATION AND RESEARCH, PUNE, INDIA}
\thanks{The first named author was partially supported by the SERB grant MTR/2017/000357 and CRG/2020/000223. We thank Professor Abhishek Saha for fruitful email correspondence in the initial stage of the project.  }


\maketitle
\begin{abstract}

In this paper, we study Siegel modular forms with  extra twists. We provide conditions  on the level and genus of the forms that is necessary for the existence of extra twists for Siegel modular forms.  We also give explicit examples of Siegel modular forms with extra twists that are different from the complex conjugation. 
\end{abstract}

\section{Introduction}
Let $f=\sum\limits_{n=1} a_n q^n \in S^1_k(N,\epsilon)$ be a classical {\it elliptic} normalized cuspidal newform of weight $k \geq 2$, genus $1$ and nebentypus character $\epsilon$ and $K=\Q(a_n)$ be the number field generated by the Fourier coefficients of $f$ with $[K:\Q]=d$. For all primes $l$ with $\lambda \mid l$ , we can associate a compatible system of $\lambda$ -adic Galois representation attached to this elliptic modular form $f$ by the classical result due to Shimura, Deligne, 
\[
\rho_{f,\lambda}:G_\Q\rightarrow \GL(V_\lambda)
\]
with $\mathrm{trace}(\rho_{f,\lambda}(\Frob_p))=a_p$ and $\det(\rho_{f,\lambda})=\epsilon(p)p^{k-1}$, for all primes $p \nmid lN$. Donote by $\rho_{f, l}:=\prod_{\lambda \mid l} \rho_{f,\lambda}$ and $V_l:=\prod_{\lambda \mid l}V_{\lambda}$. 

For the last fifty years, there is a considerable interest to study the image of these Galois representations, namely $\rho_{f,l}(G_{\Q})$ with $G_{\Q}$ given Krull topology and $\GL(V_{\lambda})$  usual $l$-adic topology. In turn,  this has  several important arithmetical applications including the proof of the Fermat's last theorem.  In this study, certain symmetries of modular forms called {\it  inner twists} or {\it extra twists} are important ingredients. Ken Ribet pioneered the study of modular forms of extra twists starting with the existence of the same~\cite[p. 48]{MR0441866}. In this case, the existence is guranteed by the work of Doi-Yamauchi, Birch and Koike.

 For $k=2$, Ribet and Momose \cite{MR617867} studied the algebra $\mathrm{Lie}(\rho_{f,l}(G_{\Q}))$ and showed that this is an explicit central simple algebra. 
Eknath Ghate and his collaborators (cf. ~\cite{MR2038777}, \cite{MR2146605}, \cite{MR3096563}) studied this algebra for higher weights $k \geq 2$  and they give a formulae for local algebras in terms of slopes of modular forms (cf. \cite{MR3096563} for details). In the the recent past, there is a  progress 
regarding the Galois representations associated with more general automorphic forms using the discovery of perfectoid spaces by Scholze~\cite{MR3418533}.

The aim of this paper is to extend results to the automorphic forms for the symplectic group $\Gsp_{2g}/\Q$ with genus $g \geq 2$. We extend a result of Ribet \cite[pg. 49 Theorem 5.7]{MR0441866} to automorphic forms on the group $\Gsp_{2g}/\Q$ with $g \geq 2$ and give explicit examples of Siegel modular forms with extra twists.  

Using perfectoid spaces and building on work of Harris- Taylor-Thorne-Clozel-Shin and many others, we associate to every Siegel modular forms $F$ of arbitrary genus $g$, the compatible system of $\lambda$ -adic  Galois representation of Peter Scholze~\cite[p. 1034, Theorem 5.1.4]{MR3418533}
\[
\rho_{F,\lambda} : G_{\Q} \rightarrow \GS_{2g}(K_{\lambda})
\]
  where $K=\Q(t_p)$ is the number field obtained by adjoining the Hecke eigenvalues $t_p$ of $F$ for all $p \in \N$. These Galois representations are continuous, semi-simple and they encode the Satake parameters of the automorphic representation associated to $F$. 

In our first main theorem, we encapsulate a condition on the genus $g$, weight $k$ and level $N$ of the Siegel modular form $F$ which is sufficient but not necessary for this Siegel modular form to have an extra twist. The level condition requires a high prime power divisibility.  This is not so surprising even for $g=1$, modular forms with extra twists are more the norm than the exception in the case of a high prime power.

 Denote by $\rho_l:= \prod_{\lambda \mid l} \rho_{F,\lambda} $, $\g_l=\mathrm{Lie}(\rho_{F,l}(G_\Q))$, $\mathfrak{gsp}_{2g}=\mathrm{Lie}(\Gsp_{2g})$ and 
 \[
 \mathfrak{a}_l = \{u\in \mathfrak{gsp}_{2g}(K \otimes \Q_l)| \mathrm{Tr}(u) \in \Q_l \};
 \]
 where $\mathrm{Tr}(u)$ denotes the trace of the matrix $u$. 
 Note traces as sufficient to determine the representations as being isomorphic \cite[p. 11]{MR1484415}. 
 We now state our first theorem:
 \begin{theorem}
\label{maintheoremn}
Let $F$ be  a {\it non-CM} (cf. Definition~\ref{cmdef}) Siegel Modular form of genus $g \geq 2$, weight $k$, level $N$  such that $|g-k|$ is odd.  Let $N$  be chosen such that $(\Z/N\Z)^{\times}$ has an element of order $2g$. The Siegel modular forms $F$ admits an extra twist
if and only if we then have a strict inclusion   $\mathfrak{g}_l \subsetneq \mathfrak{a}_l$  . 
\end{theorem}
The sufficiency of the conditions in the theorem above stems from Proposition~\ref{genusneigen}.  
From the Petersson inner product, $(c,\epsilon^{-1})$ is always an extra twist for the Siegel modular forms. It is natural to ask if there are extra twists of Siegel modular forms {\it different} from complex conjugation.

 In our  second theorem, we explicitly produce Siegel modular forms with extra twists that are not coming from complex conjugation. Start with two classical elliptic modular form $(f,g)$ with $f \in S^1_{k_1}(N, \epsilon)$, 
 $g \in S^1_{k_2}(N, \epsilon)$.  Using a suitable embedding
 $ \GL_2 \times \GL_2 \xhookrightarrow{ }\GS_4$, it is possible to  produce a Siegel modular form $Y(f \otimes g)$ that is a Yoshida lift of $(f,g)$.

 For two carefully chosen classical modular forms $f$ and $g$, their Yoshida lift denoted by $Y(f \otimes g)$ is a Siegel modular form with an extra twist.
 \begin{theorem}
 \label{extratwisttwo}
Let $f$ and $g$ be two classical non-CM elliptic modular forms with extra twist associated to the same Dirichlet character $\chi$ different from complex conjugation. Assume that the Yoshida lift 
$Y(f \otimes g)$ of $(f,g)$ 
exists, then $Y(f \otimes g)$ is a Siegel modular form that contains extra twists different from complex conjugation.   
\end{theorem}
The above theorem says that the group of extra twists can be really large. The method of this theorem also provides a systematic way of producing large class of examples of Siegel modular forms with extra twists.  In fact, we give few examples of explicit Siegel modular forms with  extra twists. We construct these examples by taking the Yoshida lift of two classical modular forms of the same level and same central nebentypus character. For families of modular forms, group of extra twists are studied by Conti~\cite{MR3932577} in a recent paper. 
 
  Kumar et al use the extra twists for Siegel modular forms with artimetic applications in the context of Lang-Trotter conjecture~\cite{kumar2022lang}. 
  It is worth mentioning that as of now it is very hard to compute Fourier coefficients and hence Hecke eigenvalue for Siegel modular forms for congruence subgroups even for $g=2$.
  There are several ways to produce Siegel modular forms using various lifts like symmetric cube, Saito-Kurakawa or Yoshida lifts. 
  It will be really intriguing to find {\it extra twists} of non-lifted Siegel modular forms that are not coming from complex conjugation. 
\section{Siegel Modular Forms}
\subsection{Motivation} 
Consider the case of classical elliptic modular forms. 
These are differential forms on the space obtained by the action of $\SL_2(\Z)$ on the upper half plane $\mathbb{H}$. 
Recall, the upper half plane can be expressed in terms of the group $\SL_2(R)/\SO(2)$, where $\SO(2)=U(1)$,  is the stabilizer of the point $i=\sqrt{-1}$. This is a maximal compact subgroup. The group $\SL_2(\Z)$ is the automorphism group of the lattice $\Z^2$ with the standard alternating inner product $<,>$ defined as:
$$<(a,b),(c,d)>=ad-bc.$$
We wish to generalise by taking for $g=2$ (the same generalisation works for arbitrary $g$). The lattice $(\Z^2)^2$ of rank $2 \times 2$, with basis $e_1,e_2,f_1,f_2$ provided with the symplectic form $<,>$ defined by : 
$$<e_i,e_j>=0,<f_i,f_j>=0,<e_i,f_j>=\delta_{ij}$$ with $\delta_{ij}$ being the Kronecker's delta. The symplectic group $\So_4(\Z)$ is by definition the automorphism group of the symplectic lattice 

$\So_4(\Z):=Aut(\Z^4,<,>)$. By using the basis of the $e's$ and $f's$ we can write the elements as group of matrices:
$\begin{bmatrix}
 A & B \\
 C & D
\end{bmatrix}$
where $A,B,C$ and $D$ are integral $2 \times 2$ matrices satisfying the following conditions. 

\subsection{Siegel modular group} Let $R$ be a commutative ring with $1$ and fix an integer $g \geq 2$. Denote by $I_g$ and $0_g$ be the identity and zero matrix of the ring $M_{g}(R)$ and now consider the matrix $J_g= 
\begin{bmatrix}
 0_g & I_g \\
 -I_g & 0_g 
\end{bmatrix}$. 
Denote by $\GS_{2g}(R)$ the algebraic group of symplectic similitudes with respect to $J_g$. Hence, $$\GS_{2g}(R)=\{M \in \GL_{2g}(R)|M^tJM=\mu(M)J\}.$$ The map $M \rightarrow \mu(M)$ defines a character $\mu: \GS_{2g}(R) \rightarrow R^{\times}$. We refer to $\mu$ as the similitude factor. The group $$\So_{2g}(R)=\{M \in \GL_{2g}(R)|M^tJM=J\}$$ is called the symplectic group of degree $2g$ with coefficients in $R$. Since the above set is the automorphism group of the alternating skew-symmetric form defined by $J_g$, hence $\So_{2g}(R)$ is a subgroup of $\GL_{2g}(R)$. For this article, we are mostly interested in $R=\Z$ or $\R$.

\begin{definition}\cite[Page 2]{KaupKaup+1983}
    For an open set $D \subset \C^{g}$, a function $F:D \rightarrow{\C}$ is called holomorphic if it is continuous and if for each fixed $(z_1^0,...,z_g^0) \in D$, and each $j=1,...,g$, the function of a single variable which is determined by the assignment $z_j \rightarrow F(z_1^0,...,z_{j-1}^0,z_j,z_{j+1}^0,...,z_g)$ is holomorphic. 
\end{definition}

The upper half plane for arbitrary genus $g \geq 1$ is the set $\mathbb{H}_g=\{X+iY |X,Y \in M_g(\C), X=X^t, Y=Y^t, Y>0\}$. The condition $Y>0$ means it is positive definite as a matrix. 
We now  define the vector valued Siegel modular form in arbitrary genus $g >1$ here and subsequently state the more explicit definition in genus $g=2$.

Suppose $\rho:\GL_g \rightarrow \GL(V)$ be a finite dimensional complex representation. Then by a vector valued Siegel modular form of genus $g$, we mean a holomorphic function $F : \mathbb{H}_g \rightarrow \C$ such that $F((A\tau+B)(C\tau+D)^{-1})=\rho(C\tau+D)F(\tau)$ for all 
$\gamma=\begin{bmatrix}
A & B \\
C & D
\end{bmatrix} \in \SP_{2g}(\mathbb{\Z})$ and $\tau \in \mathbb{H}_g$. 

The vector valued Siegel modular form of genus $2$  can be described more explicitly \cite{MR709851}. The scalar valued Siegel modular forms are just a special case of that. 
For non negative integers $k$ and $j$, let $\rho_{k,j}:\GL_2(\C)\rightarrow \GL_{j+1}(\C)$ be the irreducible representation of signature $(j+k,k)$; $$\rho_{k,j}=det^k \otimes Sym^j.$$

We now define the slash operator on a function $F : \mathbb{H}_{g} \rightarrow \C$. For a matrix 
$\gamma=\begin{bmatrix}
A & B \\
C & D
\end{bmatrix}  
\in \SP_{2g}(\Z)$, and $\tau \in \mathbb{H}_{g}$, we have the action $$(F|_{k,j}(\gamma))(\tau)=(\rho_{k,j}(C\tau+D))^{-1}F((A\tau+B)(C\tau+D)^{-1}).$$ 
\begin{definition}
A holomorphic function $F:\mathbb{H}_2 \rightarrow \C$ is called a vector valued Siegel modular form with respect to the full subgroup $\SP_4(\Z)$ and weight $\rho_{k,j}$ if $F|_{k,j}[\gamma]=F$ for all $\gamma \in \SP_4(\Z)$.  
\end{definition}

The case that $j=0$, we are mainly concern with this case in this paper and such Siegel modular forms are called scalar valued. So henceforth when we say Siegel modular forms we mean scalar valued ones. 
Recall the following definition of Siegel modular forms (respectively cusp forms) with respect to a subgroup \cite{MR2468862}. 

 \begin{definition}
 Let $K$ be a subgroup of the full modular group $\SP_{2g}(\Z)$. Then $F$ is said to be a Siegel Modular Form of weight $k$ and character $\chi$ for the subgroup $K$ if the following conditions are satisfied : 
 \begin{enumerate}
     \item $F$ is a holomorphic function on $\mathbb{H}_g$. 
     \item For every matrix $M \in K$, the function $F$ satisfies $F|_k(M)=\chi(M)F$. 
 \end{enumerate}
 \end{definition}

 The set $M_g^k(K,\chi)$ denotes Siegel Modular Forms of genus $g$, weight $k$ for the character $\chi$ and subgroup $K$. 
 \begin{definition}
 We define the operator $\Phi$ on $F \in M_g^k(K,\chi)$ by $(\Phi F)(Z')=\lim_{t\to\infty} F(\begin{bmatrix}
   Z' & 0 \\
   0 & it
 \end{bmatrix})$ with $Z' \in \mathbb{H}_{g-1}$ and $t \in \mathbb{R}$.
 \end{definition}
 
 \begin{definition}
 A Siegel modular form $F \in M_g^k(K,\chi)$ is said to be a cusp form if $F$ Lies in the kernel of the $\Phi$ operator.  
 \end{definition}
 The principal congruence subgroup for the of level $N$ for the full modular group denoted by $\Gamma^{(g)}(N)$ is defined to be 
  \[
  \Gamma^{(g)}(N)=\{M\in \SP_{2g}(\Z)~|~M\equiv I_{2g} ~ (mod N)\}.
  \]
   A congruence subgroup $K$ of $\SP_{2g}(\Z)$ is any subgroup of $\SP_{2g}(\Z)$ such that $K \supset \Gamma^{(g)}(N)$ for some $N$.  We are interested  in the following two types of the congruence subgroups: $$\Gamma_0^{(g)}(N)=\{\begin{bmatrix}
     A & B \\
     C & D
   \end{bmatrix} \in \SP_{2g}(\Z)~|~ C \equiv 0 ~ (mod ~ N)\}$$ and  $$\Gamma_1^{(g)}(N)=\{\begin{bmatrix}
     A & B \\
     C & D
   \end{bmatrix} \in \SP_{2g}(\Z)~|~ C \equiv 0 ~  (mod ~ N), A,D \equiv I_g ~ (mod ~ N)\}.$$ Obvious from the definitions we have the strict inclusions $\Gamma^{(g)}(N) \subsetneq \Gamma^{(g)}_1(N) \subsetneq \Gamma^{(g)}_0(N)$.

   \section{Extra twists for Siegel modular forms}
   \label{extratwistdefinition}
   In this section, we define extra twists for Siegel modular forms and state some properties of the same. 
\begin{prop}
\cite{MR3931351}
   A Siegel modular form $F$ has the expansion
 \begin{equation}
 F(Z)=\sum_{A\in E^{2g}, A \geq 0} t(A)e^{\pi i tr(AZ)}
  \end{equation}
  where $t(A)$ denotes the coefficients of the expansion, $E^{2g}$ denotes the set of $2g \times 2g$ half integral matrices, and $\mathrm{tr}$ denotes the trace.
 \end{prop}   
 
As recalled in the introduction, we associate to every Siegel modular forms $F$ of arbitrary genus $g \geq 1$, the compatible system of $\lambda$ -adic  Galois representations~\cite[p. 1034, Theorem 5.1.4]{MR3418533}
\[
\rho_{F,\lambda} : G_{\Q} \rightarrow \GS_{2g}(K_{\lambda})
\]
  where $K=\Q(t_p)$ is the number field obtained by adjoining the Hecke eigenvalues $t_p$ of $F$ for all $p \in \N$. These Galois representations are continuous, semi-simple and they encode the Satake parameters of the automorphic representation associated to $F$. 

  Let $\Gamma=Aut(K)$ and consider the set $D:=\{ \varepsilon : G_\Q \rightarrow K^{\times} \}$ be the set of characters from $G_{\Q}$ to $K^{\times}$. Let  $V$ be the corresponding module over $K \otimes \Q_l$ for which $Aut(V)=\GS_{2g}(K \otimes \Q_l)$. 
  
  Following \cite{MR3572258}, we define the Siegel modular forms with extra twists using Galois representation. Following lemma proves that the image of the Galois representation associated to a Siegel modular form is non-abelian.
  \begin{lemma}
 The image $\rho_{F,\lambda}(G_{\Q})$ is non abelian. 
\end{lemma}  
\begin{proof}
    Let $c$ be the complex conjugation in $G_{\Q}$. Since $c^2=I$ and $\det \rho_\lambda(c)=-1$, the eigenvalues of $\rho_\lambda(c)$ are $+1,-1$. Hence the matrix $\rho_\lambda(c)$ is similar to its Jordan canonical form with $+1$'s and $-1$'s in the diagonal. So the elements of 
    $\GS_{2g}(K_{\lambda})$ which commute with it form a block diagonal subgroup T. Since the representation $\rho_\lambda$ is irreducible the image cannot be a subset of $T$. Hence there are elements in the image that do not commute with $\rho_\lambda(c)$ and so the image is non abelian. 
\end{proof}
  \begin{definition}
  A Siegel modular form  $F$ is said to have an {\it extra twist} if there exists a tuple $(\gamma, \chi_{\gamma})$ with  $\gamma \in \Gamma$, $\chi_\gamma \in D$ such that $\rho_F \cong \gamma(\rho_F) \otimes \chi_\gamma$. 
  \end{definition}
   
   We now  state some properties about extra twists.
 \begin{lemma}
   The extra twists for $\rho_F$ over $\Q$ form a group. 
 \end{lemma}
 \begin{proof}
  Let $\gamma_1, \gamma_2 \in \Gamma$ be two extra twists of Siegel modular forms. This means that $\rho_F=\gamma_1(\rho_F) \otimes \chi_{\gamma_1}$ and $\rho_F=\gamma_2(\rho_F) \otimes \chi_{\gamma_2}$. Hence $\rho_F=\gamma_1 \gamma_2(\rho_F) \otimes \chi_{\gamma_1}\chi_{\gamma_2}=\gamma_1 \gamma_2(\rho_F) \otimes \chi_{\gamma_1.\gamma_2}$ where the character $\chi_{\gamma_1.\gamma_2}$ denotes the product of the characters $\chi_{\gamma_1}$ and $\chi_{\gamma_2}$. Further the identity element, inverse are trivial and we have already cHecked closure. So the set of self twists for $\rho_F$ over $\Q$ form a group. 
 \end{proof}
 
 We prove two lemmas similar to \cite{MR3096563}. 
 
 \begin{lemma}
   For every extra twist $\gamma$, the character $\chi_\gamma$ satisfying the equivalence is uniquely determined.
 \end{lemma}
 \begin{proof} 
   Let $\chi_{\gamma_1}$ and $\chi_{\gamma_2}$ be two characters associated to the automorphism $\gamma$. Hence $\rho_F=\gamma(\rho_F) \otimes \chi_{\gamma_1}=\gamma(\rho_F) \otimes \chi_{\gamma_2}$. This means that $\gamma(\rho_F) \otimes (\chi_{\gamma_1}-\chi_{\gamma_2})=0$. If $c$ denotes complex conjugation, 
    $\rho_F(c) \neq 0,$ and hence $ \chi_{\gamma_1}=\chi_{\gamma_2}$.  
 \end{proof}

 \begin{lemma}
 The association $\delta \rightarrow \chi^{\delta}$ defines a cocycle on the group of self-twist with values in $K^{\times}$.  
 \end{lemma}
 
 \begin{proof}
 For $\gamma, \delta \in \Gamma$, the identity $\chi_{\gamma \delta} \rightarrow  \chi_\gamma \chi_\delta^{\gamma}$ shows that $\gamma \rightarrow \chi_\gamma$ is a $1-$ cocycle. Specializing to $g \in G_{\Q}$, we see that $\gamma \rightarrow \chi_{\gamma}(g)$ is a $1-$ cocycle. By Hilbert's Theorem 90, $H^1(\Gamma,E^{\times})$ is trivial and  there is an element $\alpha(g) \in E^{\times}$ such that 
 \begin{equation}
 \frac{\gamma(\alpha(g))}{\alpha(g)}=\chi_\gamma(g) 
 \end{equation}
  for all $\gamma \in \Gamma$. Clearly $\alpha(g)$ is completely determined (upto multiplication) by elements of $F^{\times}$. Varying $g \in G_{\Q}$ we obtain the well defined map $$\overline{\alpha}:G_{\Q}\rightarrow E^{\times}/F^{\times}.$$ 
 Since each $\chi_\gamma$ is a character, $\overline{\alpha}$ is a homomorphism. 
 \end{proof}
 The following proposition closely follows from \cite[Proposition (2.15)]{MR3572258}.
 \begin{prop}
 Let $\Q[Tr(Ad(\rho))]$ denote the ring generated over $\Q$ by the set $Tr(Ad(\rho)(g))$ for $g \in G_{\Q}$. Then every element of $\Q[\mathrm{Tr}(Ad(\rho))]$ 
 is fixed by all self twists for $\rho$ over $\Q$. 
 \end{prop}
 \begin{proof}
 From \cite[Proposition (2.14)]{MR3572258} we get that there is an equivalent representation to $\rho$ which we call $\rho'$ satisfying :
 \begin{enumerate}
 \item $\rho'(g)\in \GS_4(K_\rho)$ where the field $K_\rho=\Q(tr(Ad(\rho)))$. 
 \item $\gamma(\rho')=\rho' \otimes \epsilon$.
 \end{enumerate}
  For $g \in G_{\Q}$, we consider the $1$-cocycle $c_g;\Gamma_{[g]} \rightarrow L^{\times}$ defined by $c_g(\gamma)=s^{-1}((\rho')^{\gamma}\rho'(g)^{-1})$, where $s:L^{\times}\rightarrow \GS_4(L)$ denotes the scalar morphism. By Hilbert's 90, $H^{1}(\Gamma_{[\rho']},L^{\times})$ is the trivial module. Hence $c_g(\gamma)=\frac{a_g}{\gamma(a_g)}$ with some $a_g \in L^{\times}$. Then for any $g \in G$ and $\gamma \in \gamma_{[\rho']}$, we obtain 
 \[
 \rho'(g)^{\gamma}\otimes \frac{\gamma(a_g)}{a_g}=\rho'(g). 
 \]
 The identity $\gamma(\rho'(g)a_g)=\rho'(g)a_g$ shows that $\rho(g)a_g \in \GS_4(K_{\rho})$. Hence the field 
 $\Q (Tr(Ad(\rho(g))))$ is fixed by all self twists for $\rho$ over $\Q$.
 \end{proof}
   We first define the Siegel modular form without complex multiplication below which is important for the main theorems as they true hold for forms without complex multiplication. The definition is taken from \cite[Definition 2.3]{MR3572258}. 
   
   \begin{definition} \label{cmdef}
    A Siegel modular form $F$ is said to admit complex multiplication (or be cm) if there exists non trivial $\epsilon$ such that $\rho_F \cong \rho_F \otimes \epsilon^{-1}$.
   \end{definition}
   
Denote the matrix $dU_g=$ (
   $\begin{matrix} 
   d^{-1}I_g & 0\\
   0 & dI_g 
   \end{matrix} )$. Corresponding to a congruence subgroup $K$ and a character $\psi : (\Z/N\Z)^{\times} \rightarrow \C^\times$, we have the space $S_k(K,\chi,\psi)=\{u \in S_k(K,\chi) ~ | ~ F<d>=\psi(d)F\}$, where $<d>=dU_g$. For the next proposition, we assume that $g=2$. 

It is worth mentioning that \cite[Lemma 4.14]{MR2468862}, to our knowledge holds for scalar valued Siegel modular forms and hence we haven't generalised this result to vector valued Siegel modular forms here.  
In the next Proposition we see how the same result can be generalised to Siegel modular forms of higher genus putting similar conditions connecting the level and the genus. 

\begin{prop}
\label{genusneigen}
  Let $F \in S_k^g(K,\chi,\psi)$ be  a  Siegel  cusp  form  of   level  $N$ and  weight $k$. Let the level $N$ be chosen so that $(\Z/N\Z)^{\times}$ has an element of order $2g$.    Assume that  $|g-k|$  is  odd.  There  exists  $\sigma  \in  H=\ker(\psi)$, such  that the characteristic polynomial of $\sigma$  has distinct roots.   
\end{prop}

\begin{proof}
We assume $g \geq 2$ as we already know the result for $g=1$ . 
Recall that we start with an $N$ for which $(\Z/N\Z)^\times$ has an element of order $2g$, which we call $a$. So $a^g=-1$. Rewriting the equation from \cite[Lemma 4.14]{MR2468862}, we get $\chi(aU_g)^ga^{gk-g(g+1)}F=\psi(a)F$.

Hence, we deduce that $\chi(-I)(-1)^{k-g+1}F=\psi(a)F$. For $gk \in 2\Z$ (which is ensured when $F \neq 0$), this further gives  $(-1)^{gk+k-g+1}=(-1)^{k-g+1}=\psi(a)$. Given $g$ and $k$ are of different parity, we are able to find a $\sigma \in H$ having $2g$ distinct roots such that $\rho_\psi(\sigma)=1$ with $\rho_{\psi}$ as in \cite[Pg 385]{MR2112196}.
From this we conclude that $\sigma$ is an element such that it satisfies the equation $X^{2g}-1=0$. This equation has $2g$ distinct roots. 
\end{proof}

\begin{cor} 
\label{eigenvalue}
 If $g=2$, we choose $N$  such that  $(\Z/N\Z)^{\times}$ 
  has an element of order $4$. Assume that weight is odd. There  exists  $\sigma  \in H  = \ker(\psi)$,
   such  that  $\sigma$ the characteristic polynomial of $\sigma$ has distinct roots.  
\end{cor}
Furthermore, we conclude that  $\pi_N(conj)=-1=\sigma^2$. From \cite[Lemma 1.1]{MR783510}, we thus deduce again the following corollary. 
\begin{cor}\label{eigen}
We have an inclusion  $\rho_l(H) \subset GSp_{2g} (L \otimes Q_l)$ with  $L=K^{<\sigma>}$ and $L \subset \mathbb{R}$. 
\end{cor}
From  \cite[Lemma 1.1]{MR783510}, we  deduce that $\rho_l(H) \subset GSp_4 (L \otimes Q_l)$, where $L=K^{<\sigma>}$. Now as $(conj) \in \pi_N^{-1}(<\sigma>)$, $L \subset \mathbb{R}$. 

For an embedding $\sigma : K \hookrightarrow \overline{\Q}_l$, let $\rho_\sigma$ be the map 
$$\rho_\sigma : G_{\Q} \rightarrow \GS_{2g}(K \otimes \overline{\Q_l}) \rightarrow \GS_{2g}(\overline{\Q_l});$$
where the first map is given by $\rho_l$ and the latter map is induced by $\sigma$ entry wise by $\sigma(k,t):=\sigma(k)t$.
Let $\mathfrak{gsp}_{2g}({\overline{\Q}_{l,\sigma}})=\mathfrak{g}_\sigma$ denotes the Lie algebra of the image of $\rho_\sigma(G_{\Q})$, and 
so similarly for $\tau$, we have $\mathfrak{gsp}_4(\overline{\Q}_{l,\tau})=\mathfrak{g}_\tau$. Also let 
$\mathfrak{g}_l':=\{(u,u') \in \mathfrak{g}_\sigma \times \mathfrak{g}_\tau\}$. We show that $\mathfrak{g_l'}$ is surjective onto each of its components.

We state the following  Proposition that we use in the proof of the main theorem.  The idea of this proof essentially follows  from \cite[Step 1, Pg 790]{MR457455}. 

\begin{prop}\label{surj}
   Given $\mathfrak{g}_l'$ as above, the projective maps $p_{\sigma}  :  \mathfrak{g}_l' \rightarrow \mathfrak{g}_\sigma$ and 
   $p_{\tau}  :  \mathfrak{g}_l' \rightarrow \mathfrak{g}_\tau$ are surjective. 
\end{prop}
   \begin{proof}
 Since the $\lambda$-adic Galois representation associated to Siegel modular form is irreducible,  we have $End_{\mathfrak{g}_\sigma}(V_\sigma)=\overline{\Q_l}$. Also $\g_\sigma$ is reductive with a centre that is diagonalisable as $G_\Q$ acts semi simply on $V$. It follows that $\mathfrak{g}_\sigma \subset \mathfrak{gsp}_{2g}(V_\sigma)$. Now $\mathfrak{gsp}_{2g}(V_\sigma)=\mathfrak{sp}_{2g}(V_\sigma) \oplus \overline{\Q_l}^{\times}$. Now as $\mathfrak{sp}_{2g}(V_\sigma)$ is simple, and $\g_\sigma$ is semisimple, $\g_\sigma=\mathfrak{sp}_{2g}(V_\sigma)$ or $\g_\sigma=\mathfrak{gsp}_{2g}(V_\sigma)$. The former is impossible as $\chi(Frob_l)\neq 1$ and hence $\g_\sigma=\mathfrak{gsp}_{2g}(V_\sigma)$.
   \end{proof}

   \label{Goursat}Goursat's lemma states that if $G_1$ and $G_2$ are groups such that $H$ is a subgroup of $G_1 \times G_2$, such that the two projections $p_1: H \rightarrow G_1$ and 
   $p_2: H \rightarrow G_2$ are surjective. If $N_1$ is the kernel of $p_2$ and $N_2$, the kernel of $p_1$, then $H$
   is the graph of an isomorphism from $G_1/N_1 \cong G_2/N_2$.

We now prove a lemma about the kernels of the Goursat's lemma in our context.  For each embedding $\sigma :E \rightarrow \overline{\Q_l}$, define $G_\sigma=\GS_{2g}(V_\sigma)$ and $G_l :=\prod_{\sigma:E \rightarrow \overline{\Q_l}}\GS_{2g}(V_\tau)$. Let $J=\{(u,u') \in \GS_{2g}(V_\sigma) \times \GS_{2g}(V_\tau)\}$ where $u=\rho_\sigma(g)$ and $u'=\rho_\tau(g)$ for some $g \in G_{\Q}$. Hence $J \subset G \times G'$ on which we can apply Goursat's lemma as conditions of the same are satisfied due to Proposition\ref{surj}. 
\begin{lemma}
\label{graph}
    With $G_\sigma,G_\tau$ and $J$ as above, and if $G_l \subsetneq A_l$, then $J$ is the graph of an isomorphism $G_\sigma \rightarrow G_\tau$ for some choice of $\sigma$ and $\tau$. 
\end{lemma}
\begin{proof}
    We know that 
    \[
    \GS_{2g}(V_\sigma)=\SP_{2g}(V_\sigma) \bigoplus \overline{\Q_l}^{\times}, \GS_{2g}(V_\tau)=\SP_{2g}(V_\tau) \bigoplus \overline{\Q_l}^{\times}.
    \]
We show that the kernel of the homomorphism is trivial. The kernel is an normal subgroup of $\SP_{2g}(V_\sigma) \bigoplus \overline{\Q_l}^{\times}$. 
Recall $\SP$ is a simple group. Following \cite[p. 106]{MR617867}, recall that if you resrict to certain open subgroup, the Galois representation remains irreducible. Hence, 
$G_{\sigma}$ contains $\SP_{2g}(V_\sigma)$. 
However, the determinant which is an open map.  
Hence the only possibilities for the second component of the kernel can be trivial or whole of $\overline{\Q_l}^{\times}$. Since we have the assumption that $G_l \subsetneq A_l$, there exists at least one embedding $\sigma$ for which $G_\sigma \neq A_\sigma$. Hence for that $\sigma$, the normal subgroup $N_\sigma=\{\pm I\} \times 1$. So the graph is from $G_\sigma/N_\sigma \rightarrow G_\tau/N_\tau$ where for the sake of isomorphism, both $N_\sigma$ and $N_\tau$ are trivial because for any other choice of $N_\tau$, the isomorphism would not exist. 
\end{proof}

\section{The Main Theorem} 
  In this section we prove our main Theorem~\ref{maintheoremn} for general $g \geq 1$. By extension of scalars, we regard $\rho_l$ as the $K \otimes \overline{\Q_l}$ representation, 
$$\rho_l : G_{\Q} \rightarrow \GS_{2g}(K \otimes \Q_l) \rightarrow \GS_{2g}(K \otimes \overline{\Q_l})$$
Changing bases if necessary we proved in Corollary \ref{eigen}, that there exists $H$ such that $\rho_l(H) \subset \GS_{2g}(L \otimes \Q_l)$.
Here we assume that $N\in \mathbb{N}$ such that $(\Z/N\Z)^{\times}$ has an element of order $4$.
Our aim is to calculate Lie algebra of $\rho_l(G_\Q)$ which we have denoted as $\mathfrak{g_l}$. 
On $H$, $\rho_\sigma = \rho_\tau \iff \sigma|_L=\tau|_L$.  We list the following properties of the $l$-adic Galois representation: 
\begin{enumerate}
\item 
 $\rho_l$ is semisimple. 
\item
 $det(\rho_l)=\chi(l)^{2k-3}$ 
\item 
No representation $\rho_\lambda$ with $\lambda | l$  becomes abelian on an open subgroup of $G_\Q$. 
\end{enumerate}

Because of the aforementioned conditions, $\rho_l$ satisfies the conditions for  \cite[Theorem 4.4.10]{MR457455}.

We now prove our main Theorem~\ref{maintheoremn}. The main theorem says if certain conditions are satisfied then the Siegel modular forms contain extra twist if and only if 
$\g_l \subsetneq \mathfrak{a_l}$. In other word, proving main theorem is equivalent to showing the following statements are equivalent as in the classical cases of elliptic modular forms:
\begin{enumerate}
\item
$\g_l \subsetneq \mathfrak{a_l}$
\item
$\exists ~ \sigma, \tau$ s.t. $\sigma|_K \neq \tau|_K$ but there $\exists$ an open subset $H_0$ of $G_\Q$ ~ such that 
$\rho_\sigma : H_0 \rightarrow \GS_{2g}(\overline{\Q_l})$ and $\rho_\tau : H_0 \rightarrow \GS_{2g}(\overline{\Q_l})$ are isomorphic. 
\item
There exists a finite order character $~ \phi : G_\Q \rightarrow \overline{\Q_l}^\times $  s.t.  $\rho_\tau \cong \rho_\sigma \otimes \phi$. 
\end{enumerate}
 \begin{proof}

We establish $(1) \implies (2)$. We use the same argument as in \cite[Lemma 7]{MR387283} and modify it for our case. 
Recall that $\mathfrak{g}_l'=\{(u,u') \in \mathfrak{g}_\sigma \times \mathfrak{g}_\tau\}$.
 At Lie algebra level, $\mathfrak{g}_l' \subset \mathfrak{gsp}_{2g}(V_\sigma)\times \mathfrak{gsp}_{2g}(V_\tau)$. Suppose $\mathfrak{g}_l \subsetneq \mathfrak{a}_l$, as the projections are surjective by Proposition \ref{surj}.  By applying Goursat's Lemma \ref{Goursat}, we deduce  that $\mathfrak{g}_l'$ is the graph of an isomorphism \ref{graph} $\alpha : \mathfrak{gsp}_{2g}(V_\sigma) \rightarrow \mathfrak{gsp}_{2g}(V_\tau)$ which takes $1$ to $1$. Note that as the projections are surjective, the graph always exists. However only when is the inclusion proper are the kernels trivial.

Recall that  $\mathfrak{gsp}_{2g}(V)=\mathfrak{sp}_{2g}(V) \oplus \overline{\Q_l}^{\times}I_n$. Hence,  any automorphism of $\mathfrak{gsp}_{2g}(V)$ is determined by what it does to $\mathfrak{sp}_{2g}(V)$.
Hence our automorphism is determined by its restriction to $\mathfrak{sp}_{2g}(V)$. So by \cite{MR27764}, $\alpha(u)=f\circ u \circ f^{-1}$ where $u \in \mathfrak{gsp}_{2g}(V_\sigma)$ and $f \in V_\sigma \rightarrow V_\tau$. Here $f$ is an isomorphism of $\mathfrak{g}_l'$ modules.

Shifting to Lie group level, we deduce that $\overline{f}$ is an isomorphism of $U$ modules. 
So if $H_0=\rho_l^{-1}(U)$ where $\overline{f}$ is an isomorphism of $U$ modules, it satisfies the required conditions of (2). 

We now prove $(3) \implies (1)$. We start by assuming that there exists a finite order character $\phi$ for which $\rho_\sigma \cong \rho_\tau \otimes \phi$. Hence $det(\rho_\sigma)=det(\rho_\tau \otimes \phi)$. Since $\phi$ is only a finite order character, $det(\rho_\sigma \rho_\tau^{-1})=\phi^{2g} \subsetneq \overline{\Q_l}^{\times}$. Hence $\g_\sigma \g_\tau^{-1} \subsetneq \mathfrak{a}_\sigma \mathfrak{a}_\tau$, which gives us $\g_l \subsetneq \mathfrak{a}_l$. 

To prove $(3) \implies (2)$ is easy. We start with the assumption that there is a finite order character such that $\rho_\tau \cong \rho_\sigma \otimes \phi$. Hence the kernel of $\phi$ will be a finite index subgroup of $G_{\Q}$ which will be our $H_0$. So, as $\phi \equiv 1$ on $H_0$, $\rho_\tau|_{H_0} \cong \rho_\sigma|_{H_0}$ and
$(2)$ holds true. 

We now prove $(2) \implies (3)$. For any finite index subgroup $H_0$ of $G_\Q$, the image $\rho_\sigma(H_0)$ has commutant consisting of scalar matrices in $\GS_{2g}(\overline{\Q_l})$ by Schur's lemma. The same is also true for the representation $\rho_\tau$. Hence there exists $\phi : G_\Q \rightarrow \overline{\Q_l}^\times$ of finite order as $H_0$ being open is a finite index subgroup of $G_\Q$ such that $\rho_\sigma \cong \rho_\tau \otimes \phi$. 

\end{proof}
\section{Examples of Siegel modular forms with extra twists}
\subsection{Yoshida lifts of classical modular Forms with characters}
\label{yoshida}
We briefly recall the theory of Yoshida lifts ~ \cite{MR3365801}. This is basically coming from the embedding $\GL_2 \times \GL_2 \xhookrightarrow {} \GS_{4}$. 
Given two classical elliptic, weight $2$ modular forms $f$ and $g$ of level $N_1,N_2$ respectively (two automorphic representations for $\GL_2$) associated to the same primitive character $\chi$, the Yoshida lift $F$ is the automorphic representation for $\GS_4$. Recall that for the pair $(f,g)$, automorphic representation   $F:=Y(f \otimes g)$ is the Yoshida lift if the following conditions are satisfied: 
\begin{enumerate}
    \item The adelization of $F$ generates an irreducible cuspidal automorphic represenation $\pi_F$ of $\GS_4(\mathbb{A})$. 
    \item The local $L-$ parameter for $\pi_{F,v}$ at each place $v$ is the direct sum of the $L$-parameters for $\pi_{f,v}$ and $\pi_{g,v}$. 
\end{enumerate}

The Yoshida lift is a special case of Langlands functoriality coming from the embedding of dual groups
\[
\{(g_1,g_2) \in \GL_2(\C) \times \GL_2(\C)  |  \det(g_1)=\det(g_2))\} \rightarrow \GS_4(\C)
\]
 given by
\begin{center}
$(\begin{bmatrix}
 a & b \\
 c & d
\end{bmatrix}$, 
$\begin{bmatrix}
 a' & b' \\
 c' & d'
\end{bmatrix})$ 
$\rightarrow$
$\begin{bmatrix}
 a & 0 & b & 0\\
 0 & a' & 0 & b'\\
 c & 0 & d & 0\\
 0 & c' & 0 & d'
\end{bmatrix}$. 
\end{center}

Note that the determinant condition of the embedding is satisfied since in both cases the determinant is equal to $\chi(p)p^{k-1}$ as the elliptic modular forms are associated to the same character $\chi$. 
We now define Yoshida lifts of two given classical modular forms associated to the same character. 

Following Roberts~\cite{MR1871665}, we list below the necessary conditions for the existence of the Yoshida lifts.  
Given two classical newforms $f$ and $g$, we say the pair satisfies the conditions of a Yoshida lift if \label{list}
\begin{enumerate}
    \item The modular form $f$ is not a scalar multiple of $g$
    \item The characters of $f$ and g arise from the same primitive Dirichlet character 
    \item One of the weights is $2$ and the other weight has to be an even integer greater than $2$. 
    \item There exists a finite prime $p$ at which $\pi_{f,p}$ and $\pi_{g,p}$ are both discrete series.
\end{enumerate}

\begin{definition}
Suppose these conditions are satisfied and $f,g \in S^1_2(N,\chi)$, there is a unique representation $\Pi_{F,p}$ of $\Gsp_4(\Q_p)$ satisfying \label{directsum}$$L(\Pi_{F,p})=L(\pi_{f,p}) \oplus L(\pi_{g,p}).$$ 
In this case $F$ is said to be a Yoshida lift of $f$ and $g$.  
\end{definition}
If $f,g$ are two classical modular forms with coefficient fields $K_1, K_2$ and $\Gamma_1$ and $\Gamma_2$ are the respective group of extra twists, then we denote their Yoshida lift by $Y(f \otimes g)$. We also denote the group of extra twists for the Yoshida lift as $\Gamma_Y$
 
\begin{lemma}
If two classical modular forms $f$ and $g$ satisfy the conditions listed in \ref{list}, then $\Gamma_Y \supset  \Gamma_1 \cap \Gamma_2$. 
\end{lemma}
\begin{proof}
    Suppose $\gamma \in \Gamma_1 \cap \Gamma_2$. Now by definition this means $K_1$ and $K_2$ have to be equal for a non zero $\gamma$ to exist in the intersection. Hence their compositum is $K=K_1=K_2$.  This implies that $\gamma(\rho_f)=\rho_f \otimes \chi_{\gamma}$ and $\gamma(\rho_g)=\rho_g \otimes \chi_{\gamma}$. Because of the equation \ref{directsum}, we also directly have that, for the Yoshida lift $F=Y(f \otimes g)$, we have $\gamma(\rho_F)=\rho_F \otimes \chi_\gamma$. 
\end{proof}

In our subsequent example of Siegel modular form with extra twist, we show that it is not necessarily true that we have $\Gamma=\Gamma_1 \cap \Gamma_2$. In fact, all the examples  
we have $\Gamma_Y \supsetneq  \Gamma_1 \cap \Gamma_2$.

In this section we give an explicit examples of Siegel modular forms with extra twists. We find examples of Siegel modular forms with extra twists by taking the Yoshida lifts of two classical modular forms with appropriate, weight, character and having extra twists themselves.

\subsection{Proof of Theorem~\ref{extratwisttwo}}
In this subsection, we prove that the group of extra twists for Siegel modular forms can be very large. 
\begin{proof}
Let $f, g \in S_k(N,\chi) $ be two elliptic modular forms with extra twists different from complex conjugation. Let $K=K_f \cdot K_g$ be the compositum of the Hecke fields $K_f, K_g$ of $f$ and $g$. There exists at least one $\gamma \in Aut(K)$ such that 
\begin{align}
\label{ellipticextratwists}
\gamma(\rho_f)=\rho_f \otimes \chi_\gamma,\gamma(\rho_g)=\rho_g \otimes \chi_{\gamma}.
\end{align}
Recall that we assume the necessary condition for the existence of Yoshida lift of $(f,g)$ is satisfied and let $F=Y(f \otimes g)$ be the Yoshida lift for the pair $(f,g)$ as defined in \S ~\ref{yoshida}. We claim that $F$ is a Siegel modular form with extra twist by same $\gamma$. In other words, we need to show that there exists a character $\chi_{\gamma}$ 
such that   
\begin{align}
\label{equalitysiegel}
\gamma(\rho_F)=\rho_F \otimes \chi_{\gamma}.
\end{align}

This is an equality of two $4$ dimensional Galois representation.  By Brauer-Nesbitt theorem, this is only true if all the coefficients in the characteristic polynomial are equal.  
The characteristic polynomial of the left hand side in equation~\ref{equalitysiegel} is $(x^2-\gamma(a_p)x+p^{k-1})(x^2-\gamma(a_p')+p^{k-1})$ and that of right hand side is  $(x^2-a_p \chi_{\gamma}(p)x+p^{k-1})(x^2-a_p' \chi_{\gamma}(p)x+p^{k-1})$. The above equality follows from  equation~\ref{ellipticextratwists}.
\end{proof}
We expect that the endomorphism algebra of the conjectural motive associated to $Y(f\otimes g)$ is a direct sum of the endomorphism algebras for $f$ and $g$. 
\subsection{Examples of the phenomenon described above}
In this section, we find explicit examples of Siegel modular forms with extra twists different complex conjugations. All these examples show that the Hecke field of Yoshida lift is smaller that the compositum of the individual elliptic modular forms. 
\begin{enumerate}
    \item 
Let $f, g \in S_2^1(30,\chi)$ be two newforms with $\chi$ is a Dirichlet character on $\chi :  (\Z/{30\Z})^{\times} \rightarrow \C^\times$ such that $\chi$ is determined on its generators by $\chi(7)=-\zeta_{8}^2$ and $\chi(11)=-1$. Note we can do so because this is a four dimensional space. Since the conductor of the character is $15$, condition $(2)$ of Yoshida lifting is satisfied and we can consider the Yoshida lifting $Y(f \otimes g)$ and this is an example of an explicit Siegel modular form with an extra twist. 
In this example the classical modular forms could be chosen with the expansion $f=q+\zeta_8 q^2+ (\zeta_8^3-\zeta_8^2-1)q^3+\zeta_8^2q^4+...$
and $g=q+\zeta_8^3q^2+(\zeta_8^9-\zeta_8^6-1)q^3+\zeta_8^6q^4+...$. 
For a detailed study of this space, let $K_{f,g}$ be the compositum of the coefficient fields of $f$ and $g$. It must be said that no matter what $f$ and $g$ we choose, in any case $K_{Y(f \otimes g)} \subsetneq K_{f,g}$.

Let $\Gamma_F$ be the group of extra twists of $F$. Then it is readily evident that $\Gamma_F \supseteq \Gamma_1 \cap \Gamma_2$. For the $f$ and $g$ we have chosen, $\Gamma_1=\Gamma_2$. Not also that $K_{f,g}=\Q(\zeta_8)$. We see that in this expansion the coefficient of $q^2$ in the sum of $f+g$ is $\zeta_8+\zeta_8^3$. We claim that all of the coefficients of the sum $f+g$ belong to the field $\Q(\zeta_8+\zeta_8^3)$.

It is easy to check manually that for any power $1 \leq i \leq 8$, $\zeta_8^i+\zeta_8^{3i}$ can be written in terms of $\zeta_8+\zeta_8^3$. Hence all the sum of Hecke eigenvalues of $f$ and $g$ are in  $ \Q(\zeta_8+\zeta_8^3)$. From the definition of Yoshida lift \ref{directsum}, we get that the Hecke eigenvalues of $Y(f \otimes g) \in \Q(\zeta_8+\zeta_8^3)$. Hence this is an explicit example of when field of Hecke eigenvalues of the Yoshida lift is a proper subset of the compositum of the field of the Hecke eigenvalues of the concerned classical modular forms. 
\item
Consider the $8$ dimensional complex vector space $S_2^1(100,\chi)$ where the character $\chi$ is defined by $\chi(51)=-1$ and $\chi(77)=\frac{\mu^6-3 \mu^2}{4}$; here $\mu$ is a root of the polynomial $p(x):=x^8-7x^4+16=0$. Such a character $\chi$ has conductor $20$ and the order of the group of inner twists is $8$.

Let $\mu_1:=\frac{7+\sqrt{15}i}{2},\mu_2:=-(\frac{7+\sqrt{15}i}{2})$ be the root of $p(x)$. For $i=1,2$, we have the Fourier expansions 
\[
f_i=q+\mu_iq^2+\frac{3 \mu_i^7-13 \mu_i^3}{8}q^3+....
\]
The Hecke eigenvalue ring $K_{f_1f_2}$ is of dimension $8$ over $\Q$. The Hecke eigenvalue ring is determined by the coefficients of $1,q$ and $q^2$ itself. Hence by Definition \ref{directsum}, the coefficient ring, $K_{Y(f_1 \otimes f_2)}$ is determined by the sum of Hecke eigenvalues of $q$ and $q^2$ under the two roots. By a small calculation, we see that this field happens to be $\Q(\sqrt{15}i)$ which is of extension degree $2$ over $\Q$. Hence $K_{f_1f_2}$ is of extension degree $4$ over $K_{Y(f_1 \otimes f_2)}$.
In these case also $K_{Y(f_1 \otimes f_2)} \subsetneq K_{f_1,f_2}$.
\end{enumerate}
We give two examples only but from {\tt LFMFDB} it is evident that the group of extra twists for Yoshida lifts can be very large although it can be controlled (as expected) from the individual classical elliptic modular forms.

\bibliographystyle{crelle}
\bibliography{ronit.bib}
\end{document}